\documentclass{amsart}

\usepackage[T1, T2A]{fontenc}
\usepackage[russian, english]{babel}
\usepackage{amsmath}
\usepackage{amsfonts}
\usepackage{amsthm}
\usepackage{amssymb}
\usepackage{dsfont} %for \mathds{1}
\usepackage{todonotes}

\newcommand{\supp}{\mathop{\rm supp}\nolimits}
\newcommand{\Diam}{\mathop{\rm Diam}\nolimits}

\newcommand{\id}{\mathop{\rm Id}\nolimits}

\newcommand{\Leb}{\mathop{\rm Leb}\nolimits}

\newcommand{\R}{\mathbb{R}}
\newcommand{\s}{\mathbb{S}}
\newcommand{\N}{\mathbb{N}}
\newcommand{\Z}{\mathbb{Z}}

\newcommand{\p}{\mathbb{P}}

\newcommand{\T}{\mathbb{T}}

\newcommand{\mMG}{\mathcal{M}_{G}}
\newcommand{\mMsp}{\mathcal{M}_{X}}
\newcommand{\mS}{\mathcal{S}}
\newcommand{\mA}{\mathcal{A}}
\newcommand{\mH}{\mathcal{H}}

\newcommand{\dd}{\, \mbox{d}}
\newcommand{\ball}{\mbox{B}}
\newcommand{\mgr}{\mu}
\newcommand{\msp}{\nu}

\newcommand{\eps}{\varepsilon}
\newcommand{\1}{\mathds{1}}
\newcommand{\dH}{{\rm dist}_{\mathcal{H}}}
\newcommand{\cle}{\preccurlyeq}
\newcommand{\dC}{d_{\rm C}}
\newcommand{\tmgr}{\tilde{\mgr}}
\newcommand{\tdelta}{\tilde{\delta}}
\newcommand{\m}{\lambda} %Base measure

\newtheorem{thm}{Theorem}
\newtheorem{lm}[thm]{Lemma}
\newtheorem{prop}[thm]{Proposition}

\theoremstyle{definition}
\newtheorem{defn}[thm]{Definition}
\newtheorem{ex}[thm]{Example}

\theoremstyle{remark}
\newtheorem{rem}[thm]{Remark}

\title[Non-stationary It\^o-Kawada and Ergodic Theorems]{Non-stationary It\^o-Kawada and Ergodic Theorems for random isometries}

\author{Grigorii Monakov}

\address{Grigorii Monakov, Department of Mathematics, University of California, Irvine, CA~92697, USA}

\email{gmonakov@uci.edu}

\begin{document}

\begin{abstract}
    We consider a nonstationary sequence of independent random isometries of a compact metrizable space. Assuming that there are no proper closed subsets with deterministic image we establish a weak-* convergence to the unique invariant under isometries measure, Ergodic Theorem and Large Deviation Type Estimate. We also show that all the results can be carried over to the case of a random walk on a compact metrizable group. In particular, we prove a nonstationary analog of classical It\^o-Kawada theorem and give a new alternative proof for the stationary case.
\end{abstract}

\maketitle

\section{Introduction}

Let $A$ and $B$ be two rotations of a two-dimensional sphere $\s^2$. Then it is not hard to see that for almost every (with respect to the uniform measure) pair of rotations and every starting point $x$ the trajectory $\{ x, Ax, Bx, A^2x, ABx, BAx, \ldots \}$ will be dense in $\s^2$. It was proven in \cite{AK} that if the trajectory is dense, then it equidistributes in the following sense: for any $n \in \N$ consider the set of $2^n$ points
\begin{equation*}
    P_n = \{A^n x, A^{n - 1} B x, \ldots, B^n x\}
\end{equation*}
and a region $\Delta \subset \s^2$ bounded by a piecewise smooth curve. Then the share of points of $P_n$ laying in $\Delta$ will tend to the normalized volume of $\Delta$:
\begin{equation*}
    \lim_{n \to \infty} \frac{\# (P_n \cap \Delta)}{2^n} = \frac{\Leb(\Delta)}{\Leb(\s^2)}.
\end{equation*}

We investigate a more general setting, where random independent isometries act on a compact metric space.

\subsection{Main result}

Let $(X, d)$ be a compact metric space. Denote the Borel $\sigma$-algebra on $X$ by $\mathcal{B}$ and the space of all Borel probability measures on $X$ by $\mMsp$.

Let $G$ be the group of isometries of $X$ and $\m$ be a Borel measure on $X$, such that any $g \in G$ preserves $\m$.\footnote{Such measure exists for any locally compact metric space, see \cite[Theorem 441H]{F}. For a sufficient condition for the uniqueness of such measure see Remark \ref{cor:uniq}} We equip $G$ with the usual $C^0 (X, X)$ topology, that can be defined by the following distance:
\begin{equation*}
    \dC(g, h) = \sup_{x \in X} d(g(x), h(x)) = \max_{x \in X} d(g(x), h(x)) = \max_{x \in X} d(g^{-1}(x), h^{-1}(x))
\end{equation*}
for $g, h \in G$. Notice that $G$ is compact with respect to $\dC$. We will denote by $\mathcal{B}_{G}$ the Borel $\sigma$-algebra on $G$.

Let $\mMG$ denote the space of all Borel probability measures on $G$. We will use the following
\begin{defn} \label{def:strictAper}
    We say that a Borel measure $\mgr \in \mMG$ satisfies \textbf{no deterministic images} condition if there does not exist a pair of proper closed subsets $A, B \subset X$, such that
    \begin{equation} \label{transCond}
        g(A) = B \quad \text{for any $g \in \supp(\mgr)$.}
    \end{equation} %\todo{Reference and remark about groups and abelian case}
\end{defn}

%\begin{rem}
%    Usually the term \textbf{strictly aperiodic} is used for a measure $\mgr$ on a topological group $G$ assuming that $G$ acts on itself. In that case \textbf{strict aperiodicity} means that $\supp(\mgr)$ is not contained in any coset of any subgroup of $G$, that is not equal to $G$ (see \cite{Bo1}). If one considers $X = G$ and manages to find a shift-invariant metric (see Section \ref{subsec:rwOnGroups}), then the classic \textbf{strict aperiodicity} is equivalent to our Definition \ref{def:strictAper}.
%\end{rem}

The two main result of this paper are formulated below.

\begin{thm} \label{thm:mainWeakConv}
    Consider a compact (in weak-* topology) $K \subset \mMG$, such that any $\mgr \in K$ satisfies \textbf{no deterministic images} condition. Let $\mgr_1, \mgr_2, \ldots$ be a sequence of measures in $K$. Then for any starting measure $\msp \in \mMsp$ we have
    \begin{equation*}
        \lim_{n \to \infty} \mgr_{n} * \mgr_{n - 1} * \ldots * \mgr_{1} * \msp = \m
    \end{equation*}
    with respect to the weak-* topology on $\mMsp$.
\end{thm}

\begin{rem}
    In fact, the convergence in Theorem \ref{thm:mainWeakConv} is uniform with respect to the sequence $\mgr_1, \mgr_2, \ldots$ and starting measure $\msp$ as long as the compact set $K$ is fixed (see Theorem \ref{thm:limitMeas}).
\end{rem}

\begin{rem} \label{cor:uniq}
    Theorem \ref{thm:mainWeakConv} implies that if there is at least one measure $\mgr \in \mMG$ that satisfies the \textbf{no deterministic images} condition then the measure $\m \in \mMsp$ that is invariant with respect to all isometries has to be unique. It easily follows from here that $\m$ is unique as soon as $G$ acts transitively on $X$ (every orbit is dense, or, equivalently, there is at least one dense orbit). %\todo{In fact, it's true as soon as $G$ acts transitively on $X$, decide where to say that}
\end{rem}

\begin{rem}
    Theorem \ref{thm:mainWeakConv} is similar to the classical It\^o-Kawada Theorem, for details see Section \ref{sec:rwOnGroups}.
\end{rem}

Using the recent result about abstract nonstationary Ergodic Theorem (see \cite{GK} and Theorem \ref{thm:nonstErg} below) we also deduce the following Ergodic Theorem and Large Deviation Type Estimate:

\begin{thm} \label{thm:mainErg}
    Consider a compact (in weak-* topology) $K \subset \mMG$, such that every $\mgr \in K$ satisfies \textbf{no deterministic images} condition. Let $\mgr_1, \mgr_2, \ldots$ be a sequence of measures in $K$. Then for any continuous observable $\varphi \in C(X, \mathbb{R})$ and any $x \in X$ almost surely
    \begin{equation} \label{isomErg}
        \frac{1}{n} \sum_{k = 0}^{n - 1} \varphi(g_k g_{k - 1} \ldots g_1 (x)) \to \int_{G} \varphi \dd \m \ \ \text{\rm as}\ \ \ n\to \infty,
    \end{equation}
    where the elements $g_1, g_2, \ldots$ are independent random elements in $G$ and $g_j$ is distributed with respect to $\mgr_j$.

    Moreover, an analogue of the Large Deviations Theorem holds. Namely, for any $\eps > 0$ there exist $C, \delta > 0$ such that for any $x \in X$ and any $n \in \N$
    \begin{equation} \label{isomLD}
        %\forall n \in \mathbb{N}, \quad
        \p \left( \left| \frac{1}{n} \sum_{k = 0}^{n - 1} \varphi(g_k g_{k - 1} \ldots g_1 (x)) - \int_{G} \varphi \dd \m \right| > \eps \right) < C \exp(-\delta n),
    \end{equation}
    where
    \begin{equation*} %\label{eq:Prob}
        \p = \prod_{n=1}^\infty \mgr_n.
    \end{equation*}
\end{thm}

A huge class of examples where Theorems \ref{thm:mainWeakConv} and \ref{thm:mainErg} can be applied is formed by random walks on compact topological groups. For a detailed discussion of that setting see Section \ref{sec:rwOnGroups}. Other applications include: random rotations of a unit sphere in $\R^n$ or, more generally, isometries of compact Riemannian manifolds (see \cite{GN}), random isomorphisms of graphs, isometries of shift spaces, etc. Our proofs rely heavily on compactness of the space $X$. For results about random isometries of noncompact spaces see \cite{GNe, LV}.

Ergodic Theorems and Large Deviations for general stationary and nonstationary Markov processes are objects of intensive studies in probability and random dynamical systems. For various results about the Strong Law of Large Numbers (another name for the Ergodic Theorem) see \cite{Gi, LL1, LL2, %LY,
Y1, Y2, ZYW} and references therein. Large Deviations in random dynamical systems are studied in \cite{CDK}. Large Deviation Type Estimates for nonhomogeneous Markov chains can be found in \cite{DS}. A recent monograph \cite{DSa} investigates various limit theorems including Large Deviations Theorem for nonstationary Markov processes. Finally, an abstract version of ergodic theorem that we use as a black box to establish our result (see Theorem \ref{thm:nonstErg} below) was recently obtained in \cite{GK}.

\subsection{Weak-* limit measure}

For what follows we need to use the Wasserstein distance. Slightly abusing notation we will denote the Wasserstein metric on both spaces $\mMsp$ and $\mMG$ by $W$. Recall that $W$ can be defined as follows:

\begin{defn} \label{def:Wass}
    %\todo{Add remark about two metrics on $G$ and $X$.}
    Let $\msp_1, \msp_2$ be two probability measures in $\mMsp$ (or $\mMG$). Then the \emph{Wasserstein distance} between them is defined as
    \begin{equation*}
        W(\msp_1, \msp_2) = \inf_{\gamma} \iint_{X \times X} d(x, y) \dd \gamma(x, y) \quad \left( \text{or } \inf_{\gamma} \iint_{G \times G} d(x, y) \dd \gamma(x, y) \right),
    \end{equation*}
where the infimum is taken over all probability measures $\gamma$ on $(X \times X, \sigma( \mathcal{B} \times \mathcal{B} ))$ (or $(G \times G, \sigma( \mathcal{B}_{G} \times \mathcal{B}_{G} ))$) with the marginals (projections on the $x$ and $y$ coordinates) $P_x(\gamma) = \msp_1$ and $P_y(\gamma) = \msp_2$.
\end{defn}

\begin{rem}
    It is well known (see, for example \cite[Theorem 6.9]{Vi}) that convergence in Wasserstein metric is equivalent to the one in weak-* topology.
\end{rem}

In order to prove Theorem \ref{thm:mainErg} we need the following uniform version of Theorem \ref{thm:mainWeakConv}, that might be of interest by itself:

\begin{thm} \label{thm:limitMeas}
    Consider a compact (in weak-* topology) $K \subset \mMG$, such that any $\mgr \in K$ satisfies \textbf{no deterministic images} condition. Then for any $\eps > 0$ there exists $m \in \N$, such that for any sequence of measures $\mgr_1, \mgr_2, \ldots, \mgr_m \in K$ and any starting Borel probability measure $\msp$ on $X$ the following inequality holds:
    \begin{equation} \label{mainIneq}
        W(\mgr_{m} * \mgr_{m - 1} * \ldots * \mgr_{1} * \msp, \m) < \eps.
    \end{equation}
\end{thm}

Some of the proofs in this work repeat those from our previous paper that deals with the abelian case (see \cite{M}). We, however still include them here for the convenience of the reader.

\section{Random walks on compact groups} \label{sec:rwOnGroups}

%\todo[inline]{Say that any transitive action is a factor of the action on itself.}

In this section we consider the case of a compact topological group acting on itself by left shifts. In that setting both bi-invariant metric $d$ and the bi-invariant probability measure $\m$ arise naturally from the group structure. A bi-invariant metric exists for any compact second countable topological group, see \cite[Chapter II, Theorem 8.3]{HR}. The unique regular bi-invariant Borel probability measure $\m$ on $G$ is called Haar measure and also exists for any compact topological group (see \cite[Chapter XI]{Ha}).

As we can see from the previous paragraph, the setting of the group of isometries $G$ acting on a metric compact set $X$ is formally more general than the one of a second countable compact group acting on itself. However, if the group of isometries $G$ acts transitively on $X$ (all orbits are dense) then that action is a topological factor of the action of $G$ on itself (see \cite[Section 2.5]{G}). To have at least one measure in $\mMG$ that satisfies \textbf{no deterministic images} condition the action of $G$ on $X$ has to be transitive. Hence in our context theorems in these two settings are equivalent.

Limit theorems for random walks on compact groups have been studied intensively since 1940s. Generalizing results from \cite{L} for a random walk on $\R/\Z$ several convergence results for the sequence $\mgr^{*n}$ and various forms of Ergodic Theorem were obtained in the following works: \cite{KI, Ka, PRU, U, Mu}. In \cite{Kl} one will find a uniform stationary version of Ergodic Theorem and first results on the weak-* convergence of a nonstationary sequence $\mgr_{n} * \mgr_{n - 1} * \ldots * \mgr_{1}$, namely, with a proper choice of constant $g_n \in G$ the sequence $g_n \cdot [\mgr_{n} * \mgr_{n - 1} * \ldots * \mgr_{1}]$ converges. However, the nature of the limit measure and the sequence $g_n$ is not specified. In \cite{S} an exact criterion of weak-* convergence in the stationary case is proved (see Theorem \ref{thm:IK}) for compact Hausdorff (and not necessarily metrizable) groups. In \cite{BE} it is proved (in noncommutative case) that Ergodic Theorem for a stationary random walk is equivalent to the smallest closed subgroup containing $\supp(\mgr)$ being equal to the whole group (such measures $\mgr$ are usually referred to as {\bf adapted}, see Definition \ref{def:adapted}). Results about convergence of $\mgr^{*n}$ in the sense of total variation can be found in \cite{B, RX, AG}. For a detailed survey of results about random walks on flat tori see \cite{DF}. For functional study of Markov operators associated with random walks on groups see \cite{D} and references therein. Finally, in the recent works \cite{Bo1, Bo2} the Central Limit Theorem and Law of Iterated Logarithm for the stationary random walks on groups was  established. The works mentioned above also include several nonstationary forms of the Ergodic Theorem. However, all those theorems impose some additional Doeblin-type conditions on measures $\mgr_1, \mgr_2, \ldots, \mgr_n$, which is not required for our approach. For abelian groups our results admit less technically involved proofs, see \cite{M} for details.

To formulate one of the most famous results, that is usually referred to as It\^o-Kawada Theorem we will need the following definitions (we use standard terminology, which can be found, for example, in \cite{Bo1} and \cite{D}):

\begin{defn} \label{def:adapted}
    A measure $\mgr \in \mMG$ is called \textbf{adapted} if the smallest closed subgroup containing its support $\supp(\mgr)$ is equal to $G$.
\end{defn}

\begin{rem}
    It is easy to see that for a compact group $G$ that definition is equivalent to the following: the smallest closed semi-group containing $\supp(\mgr)$ is equal to $G$.
\end{rem}

\begin{defn}
    A measure $\mgr \in \mMG$ is called \textbf{coset aperiodic} if $\supp{\mgr}$ is not contained in any coset of any proper closed subgroup of $G$.
\end{defn}

\begin{defn}
    A measure $\mgr \in \mMG$ is called \textbf{strictly aperiodic} if $\supp{\mgr}$ is not contained in any coset of any proper closed normal subgroup of $G$.
\end{defn}

\begin{rem}
    Note that while these two definitions might seem very similar, they are not equivalent. If a measure $\mgr$ is \textbf{coset aperiodic} it is obviously \textbf{strictly aperiodic}, but the converse is not true. To see that it is enough to consider any not  normal subgroup. Notice also, that \textbf{coset aperiodic} measure has to be \textbf{adapted}.
\end{rem}

The result of It\^o, Kawada, and Stromberg can be formulated as follows:
\begin{thm}[\cite{KI, S}] \label{thm:IK}
    Let $G$ be a compact topological group. Assume that $\mgr \in \mMG$ is \textbf{adapted}. Then the weak-* limit $\lim_{n \to \infty} \mgr^{*n}$ exists if and only if $\mgr$ is \textbf{strictly aperiodic}. Moreover, if the said limit exists, it is equal to the Haar measure $\m$.
\end{thm}

Theorem \ref{thm:mainWeakConv} has the following corollary, that can be considered a nonstationary version of Theorem \ref{thm:IK}:

\begin{thm} \label{thm:IKnonst}
    Let $G$ be a compact second countable topological group. Consider a compact (in weak-* topology) $K \subset \mMG$, such that any $\mgr \in K$ is \textbf{coset aperiodic}. %Then for any $\eps > 0$ there exists $n \in \N$, such that for any collection of measures $\mgr_1, \mgr_2, \ldots, \mgr_n \in K$ and
    Then for any starting measure $\msp \in \mMG$ we have
    % \begin{equation*} %\label{mainIneq}
    %     W(\mgr_{n} * \mgr_{n - 1} * \ldots * \mgr_{1} * \msp, \m) < \eps.
    % \end{equation*}
    %In particular, it follows that for any sequence of measures $\mgr_1, \mgr_2, \ldots \in K$ we have
    \begin{equation*}
        \lim_{n \to \infty} \mgr_{n} * \mgr_{n - 1} * \ldots * \mgr_1 * \msp = \m.
    \end{equation*}
\end{thm}

\begin{proof}
    Proof of that fact is a straightforward compilation of Theorem \ref{thm:mainWeakConv} and part {\bf (i)} of Lemma \ref{lm:Aper}.
\end{proof}

\begin{rem}
    A uniform version of Theorem \ref{thm:IKnonst} can also be easily obtained by combining Theorem \ref{thm:limitMeas} and part {\bf (i)} of Lemma \ref{lm:Aper}.
\end{rem}

\begin{rem}
    Notice that in Theorem \ref{thm:IKnonst} the condition of every measure $\mgr \in K$ being \textbf{coset aperiodic} cannot be replaced with every $\mgr \in K$ being \textbf{adapted} and \textbf{strictly aperiodic} (see Example \ref{ex:notAdopt}).
\end{rem}

It is also worth mentioning that an alternative prove of the classical result of It\^o and Kawada (for second countable groups) can be obtained using our methods. For an outline of the proof please see Section \ref{sec:IKproof}.

The following nonstationary Ergodic Theorem for random walks on compact metrizable groups is a direct corollary of Theorem \ref{thm:mainErg} and part {\bf (i)} of Lemma \ref{lm:Aper}:

\begin{thm}
    Consider a compact (in weak-* topology) $K \subset \mMG$, such that every $\mgr \in K$ is \textbf{coset aperiodic}. Let $\mgr_1, \mgr_2, \ldots$ be a sequence of measures in $K$. Then for any continuous observable $\varphi \in C(G, \mathbb{R})$ and any $g_0 \in G$ almost surely
    \begin{equation*}
        \frac{1}{n} \sum_{k = 0}^{n - 1} \varphi(g_k g_{k - 1} \ldots g_0 ) \to \int_{G} \varphi \dd \m \ \ \text{\rm as}\ \ \ n\to \infty,
    \end{equation*}
    where the elements $g_1, g_2, \ldots$ are independent random elements in $G$ and $g_j$ is distributed with respect to $\mgr_j$.

    Moreover, an analogue of the Large Deviations Theorem holds. Namely, for any $\eps > 0$ there exist $C, \delta > 0$ such that for any $g_0 \in G$ and any $n \in \N$
    \begin{equation*}
        %\forall n \in \mathbb{N}, \quad
        \p \left( \left| \frac{1}{n} \sum_{k = 0}^{n - 1} \varphi(g_k g_{k - 1} \ldots g_0) - \int_{G} \varphi \dd \m \right| > \eps \right) < C \exp(-\delta n),
    \end{equation*}
    where
    \begin{equation*} %\label{eq:Prob}
        \p = \prod_{n=1}^\infty \mgr_n.
    \end{equation*}
\end{thm}

\section{Proof of Theorem \ref{thm:mainErg}: Ergodic Theorem and large deviations} %\todo{Almost the whole Section is copy-paste}}

In order to formulate the result of Gorodetski and Kleptsyn that we will use to prove Theorem \ref{thm:mainErg} we will need the following

\vspace{5pt}
{\bf Standing Assumption:} {\it We will say that a sequence of distributions \\
$\mgr_1, \mgr_2, \mgr_3, \ldots$ on $C(X, X)$ satisfies the Standing Assumption if for any $\delta>0$ there exists $m\in \mathbb{N}$ such that the images of any two initial measures after averaging over $m$ random steps after any initial moment $n$ become $\delta$-close to each other in terms of Wasserstein distance:
\begin{equation*}
    W (\mgr_{n+m}*\ldots *\mgr_{n+1} *\msp, \mgr_{n+m}*\ldots *\mgr_{n+1} *\msp') < \delta
\end{equation*}
for any probability measures $\msp$ and $\msp'$ on $X$ and any $n \in \N$.
}
\vspace{5pt}

Now the nonstationary stochastic version of Birkhoff Ergodic Theorem can be formulated as follows :
\begin{thm}[{\cite[Theorem 3.1]{GK}}]\label{thm:nonstErg}
%\begin{thm}[{Theorem 3.1 from \cite{GK}}]\label{thm:nonstErg}
    Consider a sequence of measures $\mgr_1, \mgr_2, \ldots$ in $\mMG$ that satisfies the \textbf{Standing Assumption} above and a sequence of independent random maps $f_1, f_2, \ldots \in C(X, X)$, where $f_j$ is distributed with respect to $\mgr_j$. Given any Borel probability measure $\msp_0$ on $X$, define
    \begin{equation*}
        \msp_{n}:=\mgr_n*\msp_{n-1}, \quad n=1,2,\dots.
    \end{equation*}

    Then for any $\varphi\in C(X, \mathbb{R})$ and any $x\in X$, almost surely
    \begin{equation} \label{nonstErg}
        \frac{1}{n}\left|\sum_{k = 0}^{n - 1} \varphi(f_k\circ \ldots \circ f_1(x))-\sum_{k = 0}^{n - 1} \int_X \varphi \dd\msp_k \right|\to 0 \ \ \text{\rm as}\ \ \ n\to \infty.
    \end{equation}
    %where $\nu_1=\delta_x$, and $\nu_{n+1}=\mu_n*\nu_n$ for all $n\ge 1$.
    %where the measures $\msp_n$ are defined by~\eqref{eq:nu-def}.

    Moreover, an analogue of the Large Deviations Theorem holds. Namely, for any $\eps > 0$ there exist $C, \delta > 0$ such that for any $x \in X$
    \begin{equation} \label{nonstLD}
        \forall n \in \mathbb{N}, \quad \p \left( \frac{1}{n}\left|\sum_{k = 0}^{n - 1} \varphi(f_k\circ \ldots \circ f_1(x))-\sum_{k = 0}^{n - 1} \int_X \varphi \dd\nu_k \right| > \eps \right) < C \exp(-\delta n),
    \end{equation}
    where
    \begin{equation*} %\label{eq:Prob}
        \p = \prod_{n=1}^\infty \mgr_n.
    \end{equation*}
\end{thm}

Let us now show that Theorems \ref{thm:limitMeas} and \ref{thm:nonstErg} together imply Theorem \ref{thm:mainErg}.

\begin{proof}[Proof of Theorem \ref{thm:mainErg}]
    %\todo{This is copy-paste}
    It follows from inequality (\ref{mainIneq}) that for any
    $\mgr_1, \mgr_2, \ldots \in K$ the \textbf{Standing Assumption} holds. Applying Theorem \ref{thm:nonstErg} we obtain that statements (\ref{nonstErg}) and (\ref{nonstLD}) hold. Since convergence in Wasserstein metric is equivalent to weak-* convergence inequality (\ref{mainIneq}) also implies that
    \begin{equation*}
        \frac{1}{n} \sum_{k = 0}^{n - 1} \int_{X} \varphi \dd \msp_k \to \int_{X} \varphi \dd \m.
    \end{equation*}
    After that formulae (\ref{isomErg}) and (\ref{isomLD}) follow directly from (\ref{nonstErg}) and (\ref{nonstLD}).
\end{proof}

\section{Proof of Theorem \ref{thm:limitMeas}: weak-* convergence of measures}

Recall that $(X, d)$ is a metric space. In what follows we will be using following notations: for a set $A \subset X$ we will denote by $\overline{A}$ its closure and by $A^{c}$ its complement. $\ball(A, r)$ will stand for an (open) $r$-neighborhood of $A$. For $x \in X$ notation $\ball(x, r)$ will stand for $\ball(\{x\}, r)$ -- the open ball of radius $r$ centered at $x$.

The key difference between the abelian and non-abelian cases is the proof of the following

\begin{prop} \label{prop:dens}
    Consider a compact $K \subset \mMG$, such that any $\mgr \in K$ satisfies the \textbf{no deterministic images} condition. For any $\eps > 0$ there exists $m \in \N$, such that for any $\mgr_1, \mgr_2, \ldots, \mgr_m \in K$ and any $x \in X$ the set $\supp(\mgr_{m} * \mgr_{m - 1} * \ldots * \mgr_{1} * \delta_{x})$ is $\eps$-dense in $X$.
\end{prop}

Let us postpone its proof until Section \ref{sec:eps-dense} and deduce Theorem \ref{thm:limitMeas}. The proof essentially repeats the one of Theorem 8 in \cite{M} and here we present a brief version of it for the convenience of the reader.

We would like to start with a following
\begin{defn} \label{def:epsWide}
    We say that a set $Q \in \mathcal{B}$ is $\eps$-wide if it is contained in a ball of radius $\eps$ and contains a ball of radius $\eps/3$.
\end{defn}

\begin{rem}
    Note that it follows directly from Vitali covering lemma \cite{Vit} that any compact metric space can be partitioned into finitely many $\eps$-wide sets for any fixed $\eps > 0$.
\end{rem}

We will need the following corollary of Proposition \ref{prop:dens}:

\begin{lm} \label{lm:right_form}
    For any $\eps > 0$ there exists $m \in \N$ and $\delta > 0$, such that for any sequence $\mgr_1, \mgr_2, \ldots, \mgr_m \in K$, any initial measure $\msp \in \mMsp$, and any $\eps$-wide set $Q$ we have
    \begin{equation} \label{mixingCond}
        [\mgr_{m} * \mgr_{m - 1} * \ldots * \mgr_{1} * \msp] (Q) \ge \delta.
    \end{equation}
\end{lm}

\begin{proof}
    Assume that is not true and fix $\eps > 0$ for which the statement fails. By Proposition \ref{prop:dens} there exists $m \in \N$, such that for any sequence $\mgr_1, \mgr_2, \ldots, \mgr_m \in K$ and any $x \in X$ the support $\supp(\mgr_m * \mgr_{m - 1} * \ldots * \mgr_{1} * \delta_{x})$ is $\eps/8$-dense in $X$. Let us fix that $m$. Since we assume that Lemma \ref{lm:right_form} fails, for any $n \in \N$ there exists a sequence $\mgr_{1}^{(n)}, \mgr_{2}^{(n)}, \ldots, \mgr_{m}^{(n)} \in K$, a point $x_n \in X$, and an $\eps$-wide set $Q_n$, such that
    \begin{equation*}
        [\mgr_{m}^{(n)} * \mgr_{m - 1}^{(n)} * \ldots * \mgr_{1}^{(n)} * \delta_{x_n}] (Q_n) \to_{n \to \infty} 0.
    \end{equation*}
    Due to compactness of $X$ we can choose an infinite subsequence $Q_{n_k}$, such that all $Q_{n_k}$ cover the same ball $B$ of radius $\eps/4$ (see Definition \ref{def:epsWide}). After that we can choose a convergent subsequence of $x_{n_k} \to x$ and then pass to a subsequence $m$ more times, so that $\mgr_{j}^{(n)} \to_{n \to \infty} \mgr_j$ for $1 \le j \le m$. Now we see that for some $\mgr_1, \mgr_2, \ldots, \mgr_m \in K$ we have
    \begin{equation*}
        [\mgr_{m} * \mgr_{m - 1} * \ldots * \mgr_{1} * \delta_{x}] (B) = 0,
    \end{equation*}
    which contradicts the fact that $\supp(\mgr_{m} * \mgr_{m - 1} * \ldots * \mgr_{1} * \delta_x)$ is $\eps/8$-dense.
\end{proof}

Before we start the proof of Theorem \ref{thm:limitMeas} let us establish the following
\begin{lm} \label{lm:semiinv}
    For any $\mgr \in \mMG$ and $\msp \in \mMsp$ one has
    \begin{equation} \label{shift_mon}
        W(\mgr * \msp, \m) \le W(\msp, \m).
    \end{equation}
\end{lm}

\begin{proof}[Proof of Lemma \ref{lm:semiinv}]
    Let $\gamma$ be a probability measure on $X \times X$, such that $P_x(\gamma) = \msp$ and $P_y(\gamma) = \m$. Define a measure
    \begin{equation*}
        \dd \tilde{\gamma} (x, y) = \int_{G} \dd\gamma(g(x), g(y)) \dd\mgr(g).
    \end{equation*}
    It is clear that $P_x(\tilde{\gamma}) = \mgr * \msp$ and $P_y(\tilde{\gamma}) = \m$. Now
    \begin{multline*}
        \iint_{X \times X} d(x, y) \dd\tilde{\gamma} (x, y) = \\
        = \iiint_{X \times X \times G} d(x, y) \dd\gamma(g(x), g(y)) \dd\mgr(g) =\\
        = \iiint_{X \times X \times G} d(g^{-1}(x), g^{-1}(y)) \dd\gamma(x, y) \dd\mgr(g) =\\
        = \iiint_{X \times X \times G} d(x, y) \dd\gamma(x, y) \dd\mgr(g) = \\
        = \iint_{X \times X} d(x, y) \dd\gamma (x, y)
    \end{multline*}
    and inequality (\ref{shift_mon}) follows.
\end{proof}

Now we are ready to prove Theorem \ref{thm:limitMeas}.
\begin{proof} [Proof of Theorem \ref{thm:limitMeas}] %\todo{That's copy-paste}

    Let us fix $\eps > 0$ and $\msp \in \mMsp$. We will first show that for $m \in \N$ and $\delta$ satisfying (\ref{mixingCond}) the measure $\msp^{(m)} := \mgr_{m} * \mgr_{m - 1} * \ldots * \mgr_{1} * \msp$ can be decomposed into a sum
    \begin{equation*}
        \msp^{(m)} %= \mgr_{m} * \mgr_{m - 1} * \ldots * \mgr_{1} * \msp
        = \msp_0 + \msp_1,
    \end{equation*}
    where $\msp_1(X) = \delta$ and
    \begin{equation} \label{msp2dist}
        W\left( \frac{\msp_1}{\msp_1(X)}, \m \right) \le \eps.
    \end{equation}
    In order to construct that decomposition we fix a partition of $X$ into $\eps$-wide sets $\{Q_j\}_{1 \le j \le k}$. Next, define a measure
    \begin{equation*}
        \dd \nu_1 (x) = \sum_{j = 1}^{k} \frac{\delta \m(Q_j)}{\msp^{(m)}(Q_j)} \1_{Q_j} (x) \dd \msp^{(m)} (x).
    \end{equation*}
    Notice that according to (\ref{mixingCond}) we have $\frac{\delta \m(Q_j)}{\msp^{(m)}(Q_j)} < \frac{\delta}{\msp^{(m)} (Q_j)} \le 1$, so $\msp_0 = \msp^{(m)} - \msp_1$ is a well-defined positive measure and $\msp_1(Q_j) = \delta \m(Q_j)$. Clearly, $\msp_1(X) = \delta$. Now let us show that (\ref{msp2dist}) holds. We would like to construct a probability measure $\gamma$ on $G \times G$, such that
    \begin{equation} \label{gammaSupp}
        \supp(\gamma)~\subset~\cup_{j = 1}^{k} Q_j \times Q_j,
    \end{equation}
    $P_x(\gamma) = \frac{\msp_1}{\delta}$, and $P_y(\gamma) = \m$. It immediately follows from (\ref{gammaSupp}) that
    \begin{multline*}
        \iint_{X \times X} d(x, y) \dd\gamma(x, y) = \sum_{j = 1}^{k} \iint_{Q_j \times Q_j} d(x, y) \dd\gamma(x, y) \le \\
        \le \eps \sum_{j = 1}^{k} \gamma(Q_j \times Q_j) = \eps.
    \end{multline*}
    The construction of required $\gamma$ can be done in the following way: for each $Q_j$ define
    \begin{equation*}
        \gamma|_{Q_j \times Q_j} = \dfrac{\frac{\msp_1|_{Q_j}}{\delta} \otimes \m|_{Q_j}}{\m(Q_j)}.% = \frac{\frac{\msp_2|_{Q_j}}{p} \otimes h|_{Q_j}}{\msp_2(Q_j)}.
    \end{equation*}
    It's easy to verify that $\gamma$ satisfies required properties and hence (\ref{msp2dist}) holds.

    Now we can renormalize the measure $\msp_0$ and apply the same procedure $r > \log_{1 - \delta} (\eps)$ times. We will obtain a decomposition
    \begin{equation*}
        \msp^{(r)} = \msp_0 + \msp_1 + \ldots + \msp_{r},
    \end{equation*}
    (with some new measures $\msp_0$ and $\msp_1$), such that
    \begin{equation} \label{mspjdist}
        W \left( \frac{\msp_j}{\msp_j(X)}, \m \right) = W\left( \frac{\msp_j}{(1 - \delta)^{j - 1} \delta}, \m \right) \le \eps \quad \text{for any $1 \le j \le r$}
    \end{equation}
    and $\msp_0(X) = (1 - \delta)^r < \eps$ (notice that inequality (\ref{mspjdist}) holds since additional iterations don't increase the distance to measure $\m$ according to Lemma \ref{lm:semiinv}). Finally, observe that
    \begin{multline*}
        W \left( \sum_{j = 0}^{r} \msp_j, \m \right) \le \msp_0(X) W \left( \frac{\msp_0}{\msp_0(X)}, \m \right) + \sum_{j = 1}^{r} \msp_j(X) W \left( \frac{\msp_j}{\msp_j(X)}, \m \right) \le \\
        \le \Diam(X) \eps + \eps = (\Diam(X) + 1) \eps
    \end{multline*}
    and since $\eps$ was chosen an arbitrary positive constant Theorem \ref{thm:limitMeas} is proven.

\end{proof}

\section{Proof of Proposition \ref{prop:dens}: $\eps$-density of the support} \label{sec:eps-dense}

Before proving Proposition \ref{prop:dens} we will carry out some necessary preparations. Let us start by proving the following topological lemma:

\begin{lm} \label{lm:immers}
    If $A \subset X$ is closed and $g(A) \subseteq A$ for some $g \in G$ then $g(A) = A$.
\end{lm}

\begin{proof}
    Consider
    \begin{equation*}
        B = \bigcap_{m \in \N} g^m(A).
    \end{equation*}
    B is a closed $g$-invariant subset of $A$. Assume there exists $x \in A \setminus B$. Since $g$ is an isometry, we have
    \begin{equation*}
        \inf_{y \in B} d(x, y) = \inf_{y \in B} d(g^m(x), g^m(y)) = \inf_{y \in B} d(g^m(x), y)
    \end{equation*}
    for any $m \in \N$. Also, by the construction of $B$, we have
    \begin{equation*}
        \lim_{m \to \infty} \inf_{y \in B} d(g^m(x), y) = 0,
    \end{equation*}
    so $\inf_{y \in B} d(x, y) = 0$, which cannot be, since $B$ is closed and $x \notin B$. Hence $A = B$.
\end{proof}

In what follows we will be working with the following object: let $\mS$ be a set of all closed subsets of $X$ and define $\dH$ -- the Hausdorff distance on $\mS$. Let us remind the definition of $\dH$ and introduce useful notation along the way. We will denote by $D$ the following asymmetric function on $\mS \times \mS$:
\begin{equation*}
    D(A, B) = \inf \{ \eps > 0 | B \subset \ball (A, \eps) \} \quad \text{for $A, B \in \mS$.}
\end{equation*}
Recall that the Hausdorff distance on $\mS$ can be defined by the following formula:
\begin{equation} \label{HausdDef}
    \dH (A, B) = \max \{ D (A, B), D (B, A) \} \quad \text{for $A, B \in \mS$.}
\end{equation}

For a point $x \in X$ and a closed set $A \subset X$ we will denote by $d(x, A)$ the distance between them defined by the following formula:
\begin{equation*}
    d(x, A) = \inf_{y \in A} d(x, y) = \min_{y \in A} d(x, y).
\end{equation*}

The following lemma will be used later:
\begin{lm} \label{lm:neighbOfInt}
    Let $\mA \subset \mS$ be an arbitrary collection of closed subsets of $X$. Then for any $\eps > 0$ there exists $\delta > 0$, such that
    \begin{equation*}
        \bigcap_{A \in \mA} \ball (A, \delta) \subset \ball \left( \left[ \bigcap_{A \in \mA} A \right], \eps \right)
    \end{equation*}
\end{lm}

\begin{proof}[Proof of Lemma \ref{lm:neighbOfInt}]
    Assume the contrary. Then there exists a sequence of points $x_n \in X$, such that for any $A \in \mA$ we have
    \begin{equation} \label{xnclose}
        d(x_n, A) \to 0 \quad \text{as $n \to \infty$}
    \end{equation}
    and
    \begin{equation} \label{xnfar}
        d \left( x_n, \bigcap_{A \in \mA} A \right) \ge \eps.
    \end{equation}
    Since $X$ is compact we can choose a subsequence $x_{n_k} \to_{k \to \infty} x$. Then (\ref{xnclose}) implies that for any $A \in \mA$ we have
    \begin{equation*}
        d(x, A) = 0,
    \end{equation*}
    hence $x \in A$. At the same time (\ref{xnfar}) gives us
    \begin{equation*}
        d \left( x_n, \bigcap_{A \in \mA} A \right) \ge \eps,
    \end{equation*}
    which brings us to a contradiction and the proof is finished.
\end{proof}

Next, we would like to define the following pseudo-metric on $\mS$
\begin{equation*}
    \mH (A, B) = \max \{ \inf_{g \in G} D (g(A), B), \inf_{h \in G} D (h(B), A) \} \quad \text{for $A, B \in \mS$.}
\end{equation*}

Formula (\ref{HausdDef}) implies that
\begin{equation} \label{HHausdIneq}
    \dH(A, B) \ge \mH(A, B).
\end{equation}
Notice also that in the definition of $\mH$ infima can be replaced by minima because $G$ is compact with respect to $\dC$:
\begin{equation} \label{minFromComp}
    \mH (A, B) = \max \{ \min_{g \in G} D (g(A), B), \min_{h \in G} D (h(B), A) \} \quad \text{for $A, B \in \mS$}.
\end{equation}
We will need the following technical
\begin{lm} \label{lm:triangle}
    $\mH$ satisfies the triangle inequality, namely, for any $A, B, C \in \mS$
    \begin{equation} \label{triang}
        \mH(A, C) \le \mH(A, B) + \mH(B, C).
    \end{equation}
    Additionally, $\mH$ is a continuous function from $\mS \times \mS$ (where the topology on $\mS$ is defined by Hausdorff distance $\dH$) to $\R$.
\end{lm}

\begin{proof}
    We will start by proving the triangle inequality. According to formula (\ref{minFromComp}) there exist $g_1, h_1, g_2, h_2 \in G$ such that
    \begin{equation*}
        \mH (A, B) = \max \{ D (g_1(A), B), D (h_1(B), A) \}
    \end{equation*}
    and
    \begin{equation*}
        \mH (B, C) = \max \{ D (g_2(B), C), D (h_2(C), B) \}.
    \end{equation*}
    Since $g_2$ is an isometry we have
    \begin{multline} \label{triang1}
        D(g_2 g_1 (A), C) \le D(g_2 g_1 (A), g_2(B)) + D(g_2 (B), C) = \\
        = D(g_1 (A), B) + D(g_2 (B), C)
    \end{multline}
    and similarly since $h_1$ is an isometry
    \begin{multline} \label{triang2}
        D(h_1 h_2 (C), A) \le D(h_1 h_2 (C), h_1(B)) + D(h_1 (B), A) = \\
        = D(h_2 (C), B) + D(h_1 (B), A).
    \end{multline}
    Taking $\max$ of both sides of inequalities (\ref{triang1}) and (\ref{triang2}) we conclude that
    \begin{multline*}
        \max \{ D (g_2 g_1(A), C), D (h_1 h_2 (C), A) \} \le \\
        \le \max \{ D(g_1 (A), B) + D(g_2 (B), C), D(h_1 (B), A) + D(h_2 (C), B) \} \le \\
        \le \mH(A, B) + \mH(B, C),
    \end{multline*}
    which proves inequality (\ref{triang}).

    The continuity follows from the triangle inequality (\ref{triang}) and the estimate (\ref{HHausdIneq}) of $\mH$ by $\dH$.
\end{proof}

Let us also introduce the following partial order on $\mS$:
\begin{equation*}
    A \cle B \quad \text{if there exists $g \in G$, such that $g(A) \subseteq B$.}
\end{equation*}
We will need the following monotonicity result:
\begin{lm} \label{lm:mon}
    Consider $A, B, C \in \mS$, such that $A \cle B \cle C$. Then
    \begin{equation*}
        \mH(A, B) \le \mH(A, C).
    \end{equation*}
\end{lm}

\begin{proof}
    Follows from the definition of $\mH$ and $\cle$.
\end{proof}

In what follows we will use the following technical result about Wasserstein distance (on $\mMG$):

\begin{lm} \label{lm:clSupp}
    For any fixed measure $\mgr \in \mMG$ and $\eps > 0$ there exists $\delta > 0$, such that for any $\tilde{\mgr}$, such that $W(\mgr, \tilde{\mgr}) < \delta$, we have $\supp(\mgr) \subset B(\supp(\tilde{\mgr}), \eps)$.
\end{lm}

\begin{proof}[Proof of Lemma \ref{lm:clSupp}]
    %\todo{That's essentially copy-paste from abelian case}
    Assume that such $\delta$ does not exist. Then for any positive $\delta > 0$ there is a measure $\tilde{\mgr}_{\delta}$ and a point $x_{\delta}$, such that $x_{\delta} \in \supp(\mgr)$ but $x_{\delta} \notin \supp(\tilde{\mgr}) + B(0, \eps)$. Then $B(x_{\delta}, \eps / 2) \cap \supp(\tilde{\mgr}) = \varnothing$ and hence $\mgr(B(x_{\delta}, \eps)) \cdot \eps / 2 \le W(\mgr, \tilde{\mgr}) < \delta$, which implies
    \begin{equation*}
        \mgr(B(x_{\delta}, \eps/2)) < \frac{2 \delta}{\eps}.
    \end{equation*}
    Taking $\delta \to 0$ and choosing a convergent subsequence of $x_{\delta}$ by compactness we find a point ${x_0 \in \supp(\mgr)}$, such that $\mgr(B(x_{0}, \eps/2)) = 0$, which contradicts the definition of $\supp({\mgr})$.
\end{proof}

Let us intoduce another useful notation: for a measure $\mgr \in \mMG$ consider a map $T_{\mgr} : \mS \to \mS$ defined by the following formula:
\begin{equation*}
    T_{\mgr} (A) = \bigcap_{g \in \supp(\mgr)} g(A).
\end{equation*}

We are going to prove some properties of $T_{\mgr}$.

\begin{lm} \label{lm:compPertr}
    For any $\mgr \in \mMG$, $A \in \mS$ and compact $C \subset X \setminus T_{\mgr} (A)$ there exists $\delta > 0$ such that for any $\tilde{\mgr} \in \mMG$ and $\tilde{A} \in \mS$ if $W(\mgr, \tilde{\mgr}) < \delta$ and $\dH(A, \tilde{A}) < \delta$ then
    \begin{equation} \label{compPertrClaim}
        C \cap T_{\tilde{\mgr} (\tilde{A})} = \varnothing.
    \end{equation}
\end{lm}

\begin{proof}[Proof of Lemma \ref{lm:compPertr}]
    Since both $T_{\mgr} (A)$ and $C$ are compact sets there exists $\eps > 0$, such that
    \begin{equation} \label{compDist}
        d(x, y) > \eps
    \end{equation}
    for any $x \in T_{\mgr} (A)$ and $y \in C$. By Lemma \ref{lm:neighbOfInt} we can choose $\tdelta > 0$ such that
    \begin{equation} \label{tMuANeighb}
        \bigcap_{g \in \supp{\mgr}} B(g(A), \tdelta) \subset B \left( \bigcap_{g \in \supp{\mgr}} g(A), \eps \right) = B\left( T_{\mgr} (A), \eps \right).
    \end{equation}
    According to Lemma \ref{lm:clSupp} there exists $\delta > 0$, such that for any $g \in \supp(\mgr)$ there exists $\tilde{g} \in \supp(\tilde{\mgr}))$, such that
    \begin{equation} \label{tildeg}
        \dC(g, \tilde{g}) < \tdelta / 2.
    \end{equation}
    Making $\delta$ smaller if necessary we can also guarantee that
    \begin{equation} \label{setDist}
        \dH(A, \tilde{A}) < \tdelta / 2.
    \end{equation}
    It remains to derive that the choice of $\delta$ described above implies
    \begin{equation*}
        C \cap T_{\tilde{\mgr} (\tilde{A})} = \varnothing.
    \end{equation*}
    Indeed, choose a point $x \in \T_{\tilde{\mgr}} (\tilde{A})$ and pick any $g \in \supp{\mgr}$. There exists $\tilde{g} \in \supp(\tmgr)$ such that estimate (\ref{tildeg}) holds. Combining it with the fact that $\tilde{g}^{-1} (x) \in \tilde{A}$ and inequality (\ref{setDist}) we derive
    \begin{equation*}
        d(g^{-1} (x), A) < \tdelta
    \end{equation*}
    for any $g \in \supp(\mgr)$. In other words,
    \begin{equation*}
        x \in \bigcap_{g \in \supp(\mgr)} B (g(A), \tdelta).
    \end{equation*}
    This fact combined with (\ref{tMuANeighb}) tells us that
    \begin{equation*}
        x \in B(T_{\mgr} (A), \eps)
    \end{equation*}
    and from the definition of $\eps$ (formula (\ref{compDist})) we deduce that $x \notin C$. Since $x$ was an arbitrary point in $T_{\tmgr} (\tilde{A})$ we obtain equality (\ref{compPertrClaim}).
\end{proof}

One of the key properties of the map $T_{\mgr}$ is the following lower semi-continuity:
\begin{lm} \label{lm:semicont}
    If a sequence of measures $\mgr_n \in \mMG$ converges to $\mgr$ in weak-* topology and $A_n$ converges to $A$ with respect to Hausdorff distance then
    \begin{equation*}
        \mH(T_{\mgr} (A), A) \le \liminf_{n \to \infty} \mH(T_{\mgr_n} (A_n), A_n).
    \end{equation*}
\end{lm}

\begin{rem}

    Notice that $\mH (T_{\mgr} (A), A)$ is not continuous with respect to $\mgr$. For example, one can consider a circle $X = \s^1$ (parametrized by $[0, 1)$) acting on itself by translations and a sequence of measures given by
    \begin{equation*}
        \mgr_n = \begin{cases}
            0, \quad \text{with probability $1 - 1/n$;} \\
            1/2, \quad \text{with probability $1/n$.}
        \end{cases}
    \end{equation*}
    Then $\mgr_n$ converges in the weak-* topology to $\mgr = \delta_{0}$, but for any $A$ with a diameter less than $1/2$ we have $T_{\mgr_{n}} (A) = \varnothing$ and $T_{\mgr} (A) = A$.

\end{rem}

\begin{proof}[Proof of Lemma \ref{lm:semicont}]
    First of all, let us notice that by triangle inequality (established in Lemma \ref{lm:triangle})
    \begin{equation*}
        \mH(T_{\mgr_n} (A_n), A_n) \ge \mH(T_{\mgr_n} (A_n), A) - \mH(A_n, A).
    \end{equation*}
    Since $\mH$ is continuous (again by Lemma \ref{lm:triangle}) we have
    \begin{equation*}
        \lim_{n \to \infty} \mH(A_n, A) = 0,
    \end{equation*}
    so it is enough to show that
    \begin{equation*}
        \mH(T_{\mgr} (A), A) \le \liminf_{n \to \infty} \mH(T_{\mgr_n} (A_n), A).
    \end{equation*}
    Notice that for any $h \in \supp(\mgr)$ we have
    \begin{equation*}
        h^{-1} (T_{\mgr} (A)) \subset A,
    \end{equation*}
    hence
    \begin{equation*}
        \mH(T_{\mgr} (A), A) = \min_{g \in G} D(g(T_{\mgr} (A)), A).
    \end{equation*}
    Also, from definition of $\mH$ it follows that
    \begin{equation*}
        \mH(T_{\mgr_n} (A_n), A) \ge \min_{g \in G} D(g(T_{\mgr_n} (A_n)), A),
    \end{equation*}
    so it is enough to verify that
    \begin{equation} \label{asymmIneq}
        D(T_{\mgr} (A), A) \le \liminf_{n \to \infty} D(T_{\mgr_n} (A_n), A).
    \end{equation}
    Let us fix positive $\eps > 0$ and choose a compact set
    \begin{equation*}
        C = X \setminus B(T_{\mgr} (A), \eps).
    \end{equation*}
    By Lemma \ref{lm:compPertr} there exists $\delta > 0$ such that if $W(\mgr, \mgr_n) < \delta$ and $\dH(A, A_n) < \delta$ then $C \cap T_{\mgr_n} (A_n)$. Hence for $n$ big enough we have
    \begin{equation*}
        T_{\mgr_n} (A_n) \subset B(T_{\mgr} (A, \eps)).
    \end{equation*}
    It follows that for $n$ big enough
    \begin{equation*}
        D(T_{\mgr} (A), A) \le D(T_{\mgr_n} (A_n), A) + \eps.
    \end{equation*}
    Since $\eps$ is an arbitrary positive number inequality (\ref{asymmIneq}) follows, which finishes the proof of Lemma \ref{lm:semicont}.
\end{proof}

Let us also define the set $\mS_{\eps} \subset \mS$, that contains all closed subsets $A$ of $X$, such that for some $x_1, x_2 \in X$ we have $\ball (x_1, \eps) \subset A$ and $\ball (x_2, \eps) \subset X \setminus A$. The space $\mS_{\eps}$ is compact with respect to the Hausdorff distance. %\todo{Should I explain that?}

The main tool used in the proof of Proposition \ref{prop:dens} is the following

\begin{lm} \label{lm:sep}
    For any fixed $\eps > 0$ there exists $r > 0$, such that for any $\mgr \in K$ and any $A \in \mS_{\eps}$ we have
    \begin{equation*}
        \mH(T_{\mgr}(A), A) > r.
    \end{equation*}
\end{lm}

\begin{proof}
    First of all, for any $\mgr$ that satisfies \textbf{no deterministic images} condition and any $A \in \mS_{\eps}$ we have
    \begin{equation} \label{posDist}
        \mH(T_{\mgr} (A), A) > 0.
    \end{equation}
    Indeed, if $\mH(T_{\mgr} (A), A) = 0$ then there exists $g \in G$ such that $g(T_{\mgr} (A)) = A$. Hence, for any $h \in \supp(\mgr)$ we have $h (A) \supset g^{-1} (A)$. Noticing that $g^{-1} (A) = g^{-1} h^{-1} (h(A))$ and applying Lemma \ref{lm:immers} we conclude that $h(A) = g^{-1} (A)$. Since the last equality holds for any $h \in \supp(\mgr)$ it contradicts assumption (\ref{transCond}).

    %The main idea of the rest of the proof is to show that the map $(\mgr, A) \to \mH(T_{\mgr} (A), A)$ is lower semi-continuous with respect to the pair $(\mgr, A)$ and product topology, and then deduce that $\mH(T_{\mgr} (A), A)$ is uniformly bounded away from zero on a compact $K \times \mS_{\eps}$ using standard topological argument.

    Assume that Lemma \ref{lm:sep} does not hold. Then there exists a sequence of measures $\mgr_n \in K$ and a sequence of sets $A_n \in \mS_{\eps}$, such that
    \begin{equation} \label{zeroLim}
        \lim_{n \to \infty} \mH(T_{\mgr_n} (A_n), A_n) = 0.
    \end{equation}
    Since both $K$ and $\mS_{\eps}$ are compact we can choose a subsequence $n_k \to \infty$, such that $\mgr_{n_k} \to_{k \to \infty} \mgr$ and $A_{n_k} \to_{k \to \infty} A$ for some $\mgr \in K$ and $A \in \mS_{\eps}$. We need $A$ to be an element of $\mS_{\eps}$ and not just $\mS$ to guarantee that $A$ is a proper subset of $X$ and formula (\ref{posDist}) applies. For said subsequence we combine formula (\ref{zeroLim}) with Lemma \ref{lm:semicont} and deduce that
    \begin{equation*}
        \mH(T_{\mgr} (A), A) = 0,
    \end{equation*}
    which contradicts inequality (\ref{posDist}) and the proof is finished.
\end{proof}

Finaly, we are ready finish that section with the

\begin{proof}[Proof of Proposition \ref{prop:dens}]

    Let us argue by contradiction. %\todo{Justify?}
    If the statement is not true, there exists $\eps > 0$ and $x \in X$, such that the complement of the $\eps$-neighborhood of\\ $\supp(\mgr_{m} * \mgr_{m - 1} * \ldots * \mgr_{1} * \delta_{x})$ lies in $\mS_{\eps}$ for any $m \in \N$:
    \begin{equation*}
        \ball (\supp(\mgr_{m} * \mgr_{m - 1} * \ldots * \mgr_{1} * \delta_{x}), \eps)^{c} \in \mS_{\eps}.
    \end{equation*}
    Note that
    \begin{equation*}
        \ball (\supp(\mgr_{m} * \mgr_{m - 1} * \ldots * \mgr_{1} * \delta_{x}), \eps)^{c} = T_{\mgr_{m} * \mgr_{m - 1} * \ldots * \mgr_{1}} \left( \ball (x, \eps)^{c} \right).
    \end{equation*}
    Now let us finish the proof of Proposition \ref{prop:dens}. Let us denote by $A$ the set $\ball(x, \eps)^{c}$ and consider the orbit $ \{T_{\mgr_{m} * \mgr_{m - 1} * \ldots * \mgr_{1}} (A)) \}$. By our assumption it is contained in the set $\mS_{\eps}$. Also, by Lemma \ref{lm:sep} for any $m \in \N$ we have
    \begin{equation} \label{sepIn}
        \mH(T_{\mgr_{m} * \mgr_{m - 1} * \ldots * \mgr_{1}} (A), T_{\mgr_{m + 1} * \mgr_{m} * \ldots * \mgr_{1}} (A)) > r > 0.
    \end{equation}
    Observe that
    \begin{equation*}
        T_{\mgr_{m + 1} * \mgr_{m} * \ldots * \mgr_{1}} (A) \cle T_{\mgr_{m} * \mgr_{m - 1} * \ldots * \mgr_{1}} (A),
    \end{equation*}
    which fact together with inequality (\ref{sepIn}) and Lemma \ref{lm:mon} implies that for any $m, n \in \N$ we have
    \begin{equation*}
        \mH(T_{\mgr_{m} * \mgr_{m - 1} * \ldots * \mgr_{1}} (A), T_{\mgr_{n} * \mgr_{n - 1} * \ldots * \mgr_{1}} (A)) > r.
    \end{equation*}
    Applying inequality (\ref{HHausdIneq}) we deduce that
    \begin{equation*}
        \dH(T_{\mgr_{m} * \mgr_{m - 1} * \ldots * \mgr_{1}} (A), T_{\mgr_{n} * \mgr_{n - 1} * \ldots * \mgr_{1}} (A)) > r.
    \end{equation*}
    which is impossible, since a compact set $\mS_{\eps}$ cannot contain an infinite $r$-separated subset. We have arrived to a contradiction that finishes the proof of Proposition \ref{prop:dens}.

\end{proof}

\section{Appendix} \label{sec:IKproof}

We will outline a new proof the classical Theorem \ref{thm:IK}. Notice that for our methods to work we need to assume that $G$ is a compact second countable (metrizable) topological group. We will start with the following
\begin{defn}
    We say that a measure $\mgr \in \mMG$ satisfies \textbf{no deterministic $\eps$-images} condition for some $\eps > 0$ if there does not exist a pair of closed subsets $A, B \subset X$, such that all four sets $A, B, A^{c}$ and $B^{c}$ contain an open ball of radius $\eps$ and
    \begin{equation*} %\label{epsTransCond}
        g(A) = B \quad \text{for any $g \in \supp(\mgr)$.}
    \end{equation*}
\end{defn}

The following technical Lemma will be useful:

\begin{lm} \label{lm:Aper}

    Let $G$ be a compact metrizable group. Then

    \begin{enumerate}

        \item[{\bf (i)}] if $\mgr \in \mMG$ is \textbf{coset aperiodic} then $\mgr$ satisfies \textbf{no deterministic images} condition.

        \item[{\bf (ii)}] if $\mgr \in \mMG$ is \textbf{adapted} and \textbf{strictly aperiodic} then for any $\eps > 0$ there exists $n \in \N$ such that $\mgr^{*n}$ satisfies \textbf{no deterministic $\eps$-images} condition.
    \end{enumerate}

\end{lm}

\begin{proof}

    \hfill
    \begin{enumerate}

        \item[{\bf (i)}] Assume $\mgr \in \mMG$ does not satisfy the \textbf{no deterministic images} condition. Then there exists a proper closed set $A \subset G$, such that
        \begin{equation} \label{detImEq}
            g_1(A) = g_2(A) \quad \text{for any $g_1, g_2 \in \supp(\mgr)$.}
        \end{equation}
        Consider a subgroup $H \le G$, defined by
        \begin{equation*}
            H = \{h \in G\, |\, h(A) = A \}.
        \end{equation*}
        It is easy to conclude from (\ref{detImEq}) that
        \begin{equation*}
            g_{2}^{-1} g_1 \in H \quad \text{for any $g_1, g_2 \in \supp(\mgr)$.}
        \end{equation*}
        It is also easy to see that $H$ is a proper closed subgroup of $G$ since $A$ is a proper closed subset. Hence
        \begin{equation*}
            \supp(\mgr) \subset g H \quad \text{for any $g \in \supp(\mgr)$}
        \end{equation*}
        and $\mgr$ is not \textbf{coset aperiodic}.

        \item[{\bf (ii)}] Consider an \textbf{adapted} measure $\mgr \in \mMG$ and $\eps > 0$, such that $\mgr^{*n}$ does not satisfy \textbf{no deterministic $\eps$-images} condition for any $n \in \N$. Then there exists an $\eps$-proper closed set $A \subset G$, such that for any $n \in \N$ and any $g_1, \ldots, g_n, \tilde{g}_1, \ldots, \tilde{g}_n \in \supp(\mgr)$ we have
        \begin{equation*}
            g_n \ldots g_1 (A) = \tilde{g}_n \ldots \tilde{g}_1 (A).
        \end{equation*}
        Take any $g \in \supp(\mgr)$ and define
        \begin{equation*}
            A_n = g^{n} (A) \quad \text{for $n \in \N \cup \{0\}$.}
        \end{equation*}
        Consider the smallest closed subspace $\mA \subset \mS_{\eps}$ that contains all $A_n$:
        \begin{equation*}
            \mA = \overline{\bigcup_{n = 0}^{\infty} A_n}.
        \end{equation*}
        $\mA$ is a compact subspace of $\mS_{\eps}$ and $A_n$ form a dense set in it. Now let us define $A_{-1}$. Notice that for any element $g$ of a compact group $G$ there exists a sequence $n_k \to \infty$, such that
        \begin{equation*}
            g^{n_k} \to_{k \to \infty} g^{-1}.
        \end{equation*}
        Take any $g \in \supp(\mgr)$ and define $A_{-1} = g^{-1} (A)$. Notice that
        \begin{equation*}
            A_{-1} = \lim_{k \to \infty} g^{n_k} (A),
        \end{equation*}
        hence
        \begin{equation} \label{prevLimit}
            A_{-1} = \lim_{k \to \infty} \tilde{g}^{n_k} (A),
        \end{equation}
        for any $\tilde{g} \in \supp(\mgr)$. It follows directly from (\ref{prevLimit}) that
        \begin{equation*}
            \tilde{g} (A_{-1}) = \lim_{k \to \infty} \tilde{g}^{n_k + 1} (A) = \lim_{k \to \infty} g^{n_k + 1} (A) = g (A_{-1}) = A.
        \end{equation*}
        Hence for any $\tilde{g} \in \supp(\mgr)$ we have
        \begin{equation*}
            \tilde{g} (A_{-1}) = A
        \end{equation*}
        Analogously we can define
        \begin{equation*}
            A_n = g^n (A) \quad \text{for $n \in \Z$.}
        \end{equation*}
        It follows in a similar manner that for any $n \in \N$ and any $g_1, \ldots, g_n \in \supp(\mgr)$ we have
        \begin{equation*}
            g_n^{-1} \ldots g_1^{-1} (A) = A_{-n}.
        \end{equation*}
        Let $F$ be the subgroup generated by $\supp(\mgr)$. Then for any $g \in F$ we can find a number $n \in \Z$, such that $g(A) = A_n$. If such number $n$ is unique then the map that takes $g$ to $n$ defines a homomorphism
        \begin{equation*}
            f: G \to \Z.
        \end{equation*}
        If such number $n$ is not unique then the trajectory $A_n$ is periodic and the map that takes $g$ to $n$ defines a homomorphism
        \begin{equation*}
            f: G \to \Z / m \Z \quad \text{for some $m \in \N$}.
        \end{equation*}
        In both cases the kernel $\ker(f)$ is a normal subgroup in $F$. Measure $\mgr$ is \textbf{adapted}, hence $F$ is dense in $G$ and hence $\ker(f)$ is normal in $G$. It is easy to see that $\overline{\ker(f)}$ is also normal in $G$. Hence $\overline{\ker(f)}$ is a proper closed normal subgroup and $\supp(\mgr) \subset g \left( \overline{\ker(f)} \right)$ for any $g \in \supp(\mgr)$.

    \end{enumerate}

\end{proof}

Now we are ready to prove Theorem \ref{thm:IK} for a metrizable groups:

\begin{proof}[Proof of Theorem \ref{thm:IK}]
    Consider an \textbf{adapted} and \textbf{strictly aperiodic} measure $\mgr \in \mMG$. As we have pointed out before, $\mgr$ need not be \textbf{coset aperiodic}, or, which is equivalent, does not have to satisfy \textbf{no deterministic images} condition, so simply applying Theorem \ref{thm:limitMeas} would not work.

    We will start by proving that for any $\eps > 0$ there exists $n \in \N$ such that $\supp(\mgr^{*n})$ is $\eps$-dense. Similarly to the proof of Proposition \ref{prop:dens} we consider
    \begin{equation*}
        A = \ball (\supp(\mgr), \eps/2)^{c}
    \end{equation*}
    If $\supp(\mgr^{*n})$ does not become $\eps$-dense then the whole trajectory $T_{\mgr^{*n}} (A)$ lies in $\mS_{\eps / 2}$. It is possible to prove that for $m$ big enough there exists $r > 0$, such that for any $A \in \mS_{\eps/2}$ we will have
    \begin{equation*}
        \mH(T_{\mgr^{*m}} (A), A) > r.
    \end{equation*}
    To observe that one needs to repeat the proof of Lemma \ref{lm:sep}, keeping in mind that thanks to part {\bf (ii)} of Lemma \ref{lm:Aper} for $m$ big enough the measure $\mgr^{*m}$ satisfies \textbf{no $\eps/2$ deterministic images} condition. So, analogously to the proof of Lemma \ref{lm:sep}, the set of points $\{T_{\mgr^{*mn}} (A)\}_{n \in \N}$ is $r$-separated and cannot lie in the compact set $\mS_{\eps/2}$.

    Hence for any $\eps > 0$ there exists $n \in \N$ such that $\supp(\mgr^{*n})$ is $\eps$-dense. The following fact can be obtained from the compactness of $\mS$:

    % Instead, we are going to establish a fact similar to Lemma \ref{lm:right_form} for the case under consideration, and after that the proof of Theorem \ref{thm:limitMeas} can be repeated without any changes. All we need to prove is the following

    \begin{lm} \label{lm:IKfix}
        For any $\eps > 0$ there exists $n \in \N$ and $\delta > 0$, such that for any $\eps$-wide set $Q$ (see Definition \ref{def:epsWide}) we have
        \begin{equation*}
            [\mgr^{*n}] (Q) \ge \delta.
        \end{equation*}
    \end{lm}

    It remains to repeat the proof of Theorem \ref{thm:limitMeas} replacing Lemma \ref{lm:right_form} with Lemma \ref{lm:IKfix}.

\end{proof}

\begin{ex}[Stromberg, see \cite{S}] \label{ex:notAdopt}
    For $G = S_{3}$ there exists a pair of measures $\mgr_1, \mgr_2 \in \mMG$, such that both $\mgr_1$ and $\mgr_2$ are \textbf{strictly aperiodic} and \textbf{adapted}, but the sequence $\msp_n \in \mMG$ given by
    \begin{equation*}
        \msp_1 = \mgr_1, \quad \msp_{2n} = (\mgr_2 * \mgr_1)^{*n}, \quad \msp_{2n + 1} = \mgr_1 * \msp_{2n}
    \end{equation*}
    does not converge in $\mMG$.

    Namely, take
    \begin{equation*}
        \mgr_1 = \frac{1}{2} \delta_{(2 3)} + \frac{1}{2} \delta_{(1 2 3)}
    \end{equation*}
    and
    \begin{equation*}
        \mgr_2 = \frac{1}{2} \delta_{(2 3)} + \frac{1}{2} \delta_{(1 3 2)}.
    \end{equation*}

    Then it is easy to compute that $\supp(\msp_{2n}) = \{\id, (1 2)\}$ and $\supp(\msp_{2n + 1}) = \{(2 3), (1 2 3)\}$.
\end{ex}

\section*{Acknowledgements}

The author is grateful to A. Gorodetski and V. Kleptsyn for fruitful discussions. The author was supported in part by NSF grant DMS--2247966 (PI: A.\,Gorodetski).

\end{document}